\documentclass{article}
\usepackage{amssymb,amsmath}
\usepackage{graphics}
\usepackage{CJK}

\def\qed{\hfill {\hbox{${\vcenter{\vbox{               %HOLLOW SQUARE
   \hrule height 0.4pt\hbox{\vrule width 0.4pt height 6pt
   \kern5pt\vrule width 0.4pt}\hrule height 0.4pt}}}$}}}

\def\tr{\triangleright}

\newtheorem{theorem}{Theorem}
\newtheorem{definition}{Definition}
\newtheorem{lemma}[theorem]{Lemma}
\newtheorem{proposition}[theorem]{Proposition}

\newtheorem{example}{Example}

\newtheorem{remark}[example]{Remark}

\newenvironment{proof}[1][Proof]{\smallskip\noindent{\bf #1.}\quad}%
{\qed\par\medskip}

\date{}

\title{\Large \textbf{Hom Quandles}}

\author{Alissa S. Crans\footnote{\texttt{acrans@lmu.edu}} 
\and Sam Nelson\footnote{\texttt{knots@esotericka.org}}} 

\begin{document}
\maketitle

\begin{abstract}
If $A$ is an abelian quandle and $Q$ is a quandle, the hom set $\mathrm{Hom}(Q,A)$ of quandle homomorphisms from $Q$ to $A$ has a natural quandle structure. We exploit this fact to enhance the 
quandle counting invariant, providing an example of links with the same 
counting invariant values but distinguished by the hom quandle structure. We 
generalize the result to the case of biquandles, collect observations
and results about abelian quandles and the hom quandle, and show that the category of abelian quandles is symmetric monoidal closed.\end{abstract}

\textsc{Keywords:} Quandles, biquandles, abelian quandles, abelian biquandles,
enhancements of counting invariants

\textsc{2010 MSC:} 57M27, 57M25

\section{\large \textbf{Introduction}}
Ronnie Brown said it best when declaring, ``One of the irritations of group theory is that the set Hom($H$,$K$) of homomorphisms between groups $H$ and $K$ does not have a natural group structure."  Of course, when $H$ and $K$ are both commutative, we know that Hom($H$, $K$) is also a commutative group.  %We need $K$ to be abelian for the product of two homomorphisms to again be a homomorphism.  
Since quandles are algebraic structures having groups as their primordial example, it is natural to wonder when, if ever, the set of quandle homomorphisms from a quandle $X$ to a quandle $X'$, $\mathrm{Hom}(X,X')$,  possesses additional structure.  It was shown
in \cite{I} that if $Q(K)$ is the fundamental quandle of a knot and $X$ is
an Alexander quandle, the set $\mathrm{Hom}(Q(K),X)$ has an Alexander quandle structure.  We 
generalize this 
result to show that the set $\mathrm{Hom}(Q(K),X)$ has a quandle structure
provided the target quandle $X$ is an `abelian', or `medial', quandle. Moreover,
for links with two or more components, the resulting quandle structure is 
not determined by the cardinality of the target quandle and is thus a 
stronger invariant than the counting invariant $|\mathrm{Hom}(Q(K),X)|$. 

This
paper is organized as follows. In Section \ref{QB} we recall the basics of
quandles, including definitions and examples. In Section \ref{AQ} we turn our focus to the case 
of abelian quandles, also called medial quandles. Then, in Section
\ref{HQ}, specifically in Proposition \ref{p:main}, we define a natural quandle structure on the set of quandle
homomorphisms from an arbitrary quandle to a finite abelian quandle.
 %In Proposition \ref{p:main} we provide necessary and sufficient conditions on quandles $Q$ and $A$ for the set Hom$(Q,A)$ of quandle homomorphisms to possess a quandle structure.  
 We continue in this section by illustrating properties of the hom quandle, \textrm{\bf Hom}($Q$,$A$), including those inherited from the quandle $A$.
In Section \ref{HQE} we apply the results of Section \ref{HQ} to 
define an enhanced invariant of links associated to finite abelian quandles and in
Section \ref{BQ} we generalize the results of previous sections to the 
case of biquandles. In Section \ref{C} we consider the category of abelian 
quandles and show that it is symmetric monoidal closed, and we conclude in Section \ref{Q} 
with some questions for future work.

\section{\large \textbf{Acknowledgements}}
We thank Tom Leinster for insightful and helpful discussions related to Section \ref{C} and James McCarron for pointing out an error in the initially posted
version.  We are also grateful for useful conversations with Lou Kauffman and David Radford when this work was in the beginning stages.  

\section{\large \textbf{Quandle Basics}}\label{QB}

We begin with a definition from \cite{J}.

\begin{definition}
\textup{A \textit{quandle} is a set $X$ equipped with a binary operation 
$\tr:Q\times X\to X$ satisfying
\begin{list}{}{}
\item[(i)]{ (idempotence) for all $x\in X$, $x\tr x = x$,}
\item[(ii)]{(inverse) for all $x,y\in X$, there is a unique $z\in X$ with $x=z\tr y$, and}
\item[(iii)]{(self-distributivity) for all $x,y,z\in X$, $(x\tr y)\tr z=(x\tr z)\tr (y\tr z).$}
\end{list}}
\end{definition}

The quandle axioms capture the essential properties of group conjugation, and correspond to the Reidemeister moves on oriented link diagrams
with elements corresponding to arcs and the quandle operation $\tr$ 
corresponding to crossings with $x\tr y$ the result of $x$ crossing under 
$y$ from right to left. In particular, the operation is distinctly 
non-symmetrical -- in $x\tr y$, $y$ is acting on $x$ and not conversely -- 
and thus it makes sense to use a non-symmetrical symbol like $\tr$. 

Axiom (i) says that every element of $X$ is idempotent under $\tr$. Axiom (ii) 
says that the action of $y$ on $X$ defined by $f_y(x)=x\tr y$ is 
bijective for every $y\in X$. Hence there are inverse actions, denoted by 
$f_y^{-1}(x)=x\tr^{-1}y$, and Axiom (ii) is equivalent to 
\begin{list}{}{}
\item[(ii$'$)] there is an \textit{inverse}, or {\it dual}, operation $\tr^{-1}:X\times X\to X$ 
satisfying for all $x,y\in X$
\[(x\tr y)\tr^{-1} y =x = (x\tr^{-1} y)\tr y.\] 
\end{list}
Thus, we can eliminate the existential quantifier in Axiom (ii) at the cost
of adding a second operation. It is a straightforward exercise to show that
the inverse operation is also idempotent and self-distributive, so $X$ is 
a quandle under $\tr^{-1}$, known as the \textit{dual quandle} of $(X,\tr)$.
One can also show that the two triangle operations distribute over each other,
i.e. we have
\[(x\tr y)\tr^{-1} z= (x\tr^{-1} z)\tr (y\tr^{-1} z)\quad\mathrm{and}\quad
(x\tr^{-1} y)\tr z= (x\tr z)\tr^{-1} (y\tr z).
\]

Axiom (iii) says that the quandle operation is self-distributive. This axiom
then implies that the action maps $f_y:X\to X$ are endomorphisms of the 
quandle structure:
\[f_z(x\tr y)= (x\tr y)\tr z = (x\tr z)\tr (y\tr z) =f_z(x)\tr f_z(y).\]
Indeed, Axioms (ii) and (iii) together say that the action of any element $y\in X$
is always an automorphism of $X$. If $(X,\tr)$ satisfies only  (ii) and (iii), then 
$X$ is called a \textit{rack} or \textit{automorphic set}; a quandle is then a rack 
in which every element is a fixed point of its own action.

A quandle is {\it involutory} if $\tr =\tr^{-1}$, i.e. if for all $x$ and $y$ 
in $X$ we have 
$(x \rhd y) \rhd y = x$.  Involutory quandles are also known as {\it kei} or
\begin{CJK*}{UTF8}{min}圭\end{CJK*}; see \cite{NP,T} for more details.

\begin{example}
\textup{Any module over $\mathbb{Z}[t^{\pm 1}]$ is a quandle with operations}
\[x\tr y=tx+(1-t)y \quad \mathrm{and}\quad 
x\tr^{-1} y=t^{-1}x+(1-t^{-1})y.\]
\textup{Such a quandle is called an \textit{Alexander quandle}.
Any module over $\mathbb{Z}[t]/(t^2-1)$ is a kei under 
\[x\tr y= tx+(1-t)y\]
known as an \textit{Alexander kei}.}
\end{example}

\begin{example}
\textup{Any vector space $V$ over a field $k$ 
is a quandle under the operations }
\[\vec{x}\tr\vec{y}=\vec{x}+
\langle\vec{x},\vec{y} \rangle \vec{y}
\quad\mathrm{and}\quad
\vec{x}\tr^{-1}\vec{y}=\vec{x}-
\langle\vec{x},\vec{y} \rangle \vec{y}\]
\textup{where $\langle,\rangle:V\times V\to k$ is an antisymmetric
bilinear form (if the characterisitic of $k$ is 2, then we also require
$\langle \vec{x},\vec{x}\rangle=0$ for all $\vec{x}\in V$). Such a quandle 
is called a \textit{symplectic quandle} \cite{NN}.}
\end{example}

As briefly mentioned already, the quandle axioms can be understood as arising from the oriented Reidemeister
moves where quandle elements are associated to arcs in a knot or link diagram
and the quandle operation $x\tr y$ is interpreted as arc $x$ crossing under arc
$y$ from the right. We note that the orientation of the undercrossing arc is not
relevant, but only the orientation of the overcrossing arc. The inverse triangle 
operation from Axiom (ii$'$) can be interpreted as the understrand crossing 
backwards from left to right as illustrated below:
\[\includegraphics{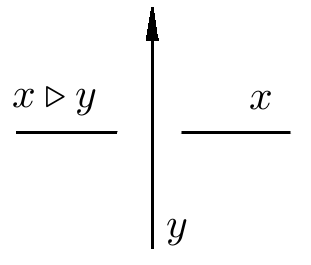}\quad \includegraphics{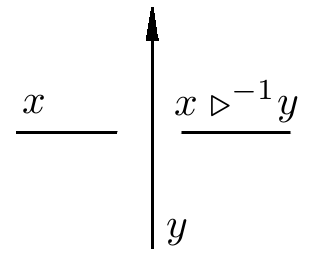}
\]

\noindent Then the quandle axioms are exactly the conditions required for the diagrams
to match up one-to-one before and after the Reidemeister moves.
\[\includegraphics{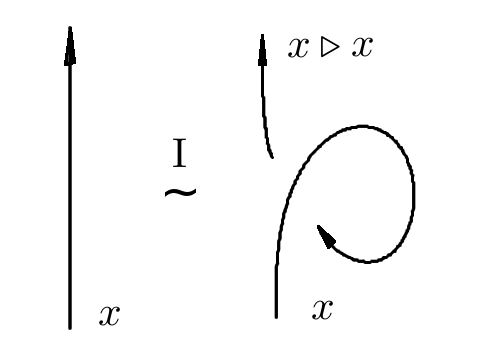} \quad \includegraphics{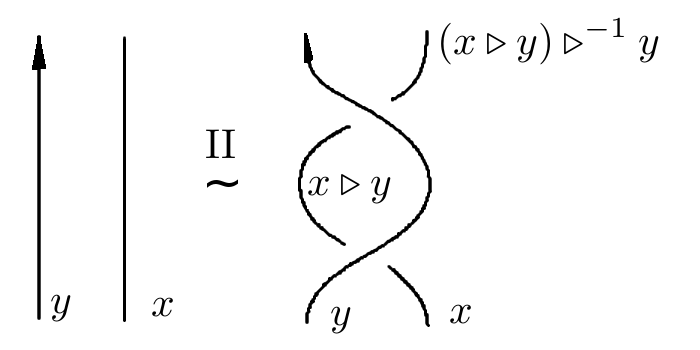}\]
\[\includegraphics{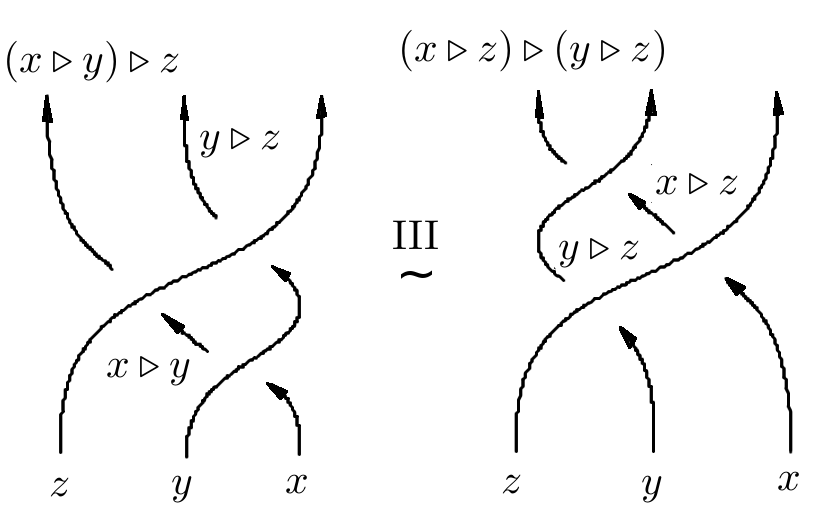}\]

Given an oriented knot or link $K$, the \textit{knot quandle}, denoted $Q(K)$, 
is the quandle with generators corresponding to arcs in a diagram of $K$
and relations given by the crossings. More precisely, the elements of the 
knot quandle are equivalence classes of quandle words in the generators
under the equivalence relation generated by the crossing relations and the
quandle axioms.

We will find it convenient to specify quandle structures on a finite sets
$X=\{x_1,x_2,\dots,x_n\}$ using an $n\times n$ matrix encoding the
quandle operation table. In particular, the entry in row $i$ column $j$ of
the quandle matrix is $k$ where $x_k=x_i\tr x_j$. Then, for example, the
Alexander quandle structure on $\mathbb{Z}_3=\{1,2,3\}$ (we use 3 for the class
of zero so we can number our rows and columns starting with 1) with
quandle operation $x\tr y= 2x+2y$ has matrix
\[M=\left[\begin{array}{ccc} 
1 & 3 & 2 \\
3 & 2 & 1 \\
2 & 1 & 3 \\
\end{array}
\right].\]

Given a knot $K$ and a finite quandle $X$, the cardinaility of the set of 
quandle homomorphisms (maps $f:Q(K)\to X$ such that $f(x\tr y)=f(x)\tr f(y)$) 
is a computable knot invariant known as the \textit{quandle counting invariant}. A 
quandle 
homomorphism $f:Q(K)\to X$ corresponds to a labeling of the arcs in a diagram
of $K$ with elements of $X$ such that the crossing relations are all satisfied.
For instance, the trefoil knot $3_1$ has nine colorings by the quandle structure
on $\mathbb{Z}_3$ listed above, as shown below:
\[\includegraphics{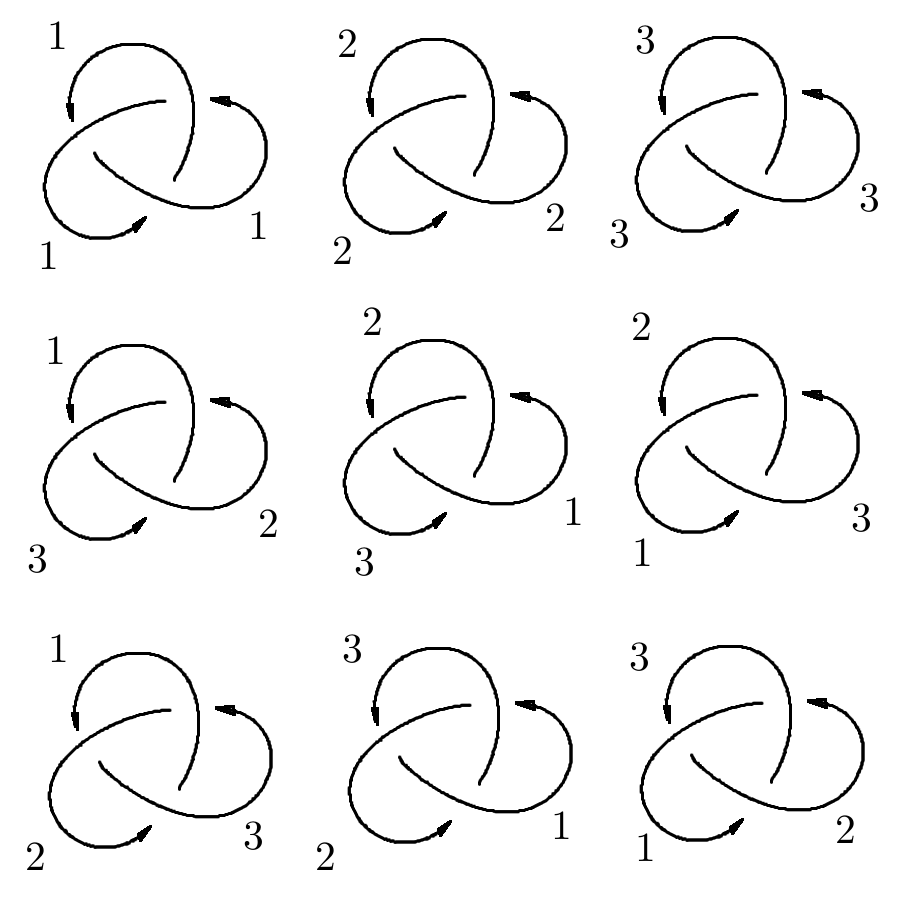}\]
Hence, the quandle counting invariant here is 
$|\mathrm{Hom}(Q(3_1),\mathbb{Z}_3)|=9$.\footnote{The reader may recognize 
these as Fox 3-colorings.}

\section{\large \textbf{Abelian Quandles}}\label{AQ}

We now turn our attention to a special class of quandles known as 
`abelian' quandles.  The reader should be aware that, unlike in the group 
case, the adjective ``abelian" is not synonymous with ``commutative."  Abelian 
quandles satisfy the condition below whereas commutative quandles satisfy 
$a \rhd b = b \rhd a$.

\begin{definition}
\textup{A quandle $Q$ is \textit{abelian} if for all $x,y,z,w\in Q$ we have}
\[(x\tr y)\tr(z\tr w) =(x\tr z)\tr(y\tr w).\]
\textup{Abelian quandles are also called \textit{medial} quandles.}
\end{definition}

\begin{example}
\textup{Alexander quandles are abelian:}
\begin{eqnarray*}
(x\tr y)\tr(z\tr w) & = & t(tx+(1-t)y)+(1-t)(tz+(1-t)w) \\
& = &  t^2x+t(1-t)y+t(1-t)z+(1-t)^2w \\
& = & t(tx+(1-t)z)+(1-t)(ty+(1-t)w) \\
& = & (x\tr z)\tr(y\tr w)
\end{eqnarray*}
\end{example}

\begin{example}
\textup{However, not all abelian quandles are Alexander. The quandle $Q_2$
with operation table 
\[\left[\begin{array}{ccc}
1 & 1 & 2 \\
2 & 2 & 1 \\
3 & 3 & 3\\ 
\end{array}\right]\]
is abelian as can be verified by checking that
$(a\tr b)\tr (c\tr d)=(a\tr c)\tr (b\tr d)$ for $a,b,c,d\in\{1,2,3\}$;
however, $Q_2$ is not isomorphic to any Alexander quandle by the following 
lemma.}
\end{example}

\begin{lemma}
If $Q$ is an Alexander quandle containing an element $y\in Q$ which acts 
trivially on $Q$, i.e. if $x\tr y=x$ for all $x\in Q$, then $Q$ is 
isomorphic to the trivial quandle on $|Q|$ elements.
\end{lemma}

\begin{proof}
Suppose $y$ acts trivially on $Q$, so that $x\tr y= x$ for all $x\in Q$. Then
\[0=x-(x\tr y)= x-tx-(1-t)y=(1-t)(x-y)\]
for all $x\in Q$. In particular, since every $x\in Q$ has the form $x=(x+y)-y$,
the map $(1-t):Q\to Q$ is the zero map, so $t:Q\to Q$ is the identity map. Then
\[x\tr z=tx+(1-t)z= 1x+0z=x\]
and the quandle operation on $Q$ is trivial.
\end{proof}

\begin{example}
\textup{Unlike Alexander quandles, symplectic quandles are generally 
non-abelian. Consider the symplectic quandle structure on $(\mathbb{Z}_2)^2$
defined by
\[\left[\begin{array}{c}
x_1 \\
x_2
\end{array}\right]
\tr
\left[\begin{array}{c}
x_1 \\
x_2
\end{array}\right]
=
\left[\begin{array}{c}
x_1 \\
x_2
\end{array}\right]+
\left(\left[\begin{array}{cc} x_1 & x_2 \end{array}\right]
\left[\begin{array}{cc}
0 & 1 \\
1 & 0
\end{array}\right]
\left[\begin{array}{c}
y_1 \\
y_2
\end{array}\right]\right)
\left[\begin{array}{c}
y_1 \\
y_2
\end{array}\right].
\]
This four-element quandle has operation matrix
\[M=\left[\begin{array}{cccc}
1 & 1 & 1 & 1 \\
2 & 2 & 4 & 3 \\
3 & 4 & 3 & 2 \\
4 & 3 & 2 & 4 \\
\end{array}\right]\]
where
\[x_1=\left[\begin{array}{c} 0 \\ 0\end{array}\right],
x_2=\left[\begin{array}{c} 1 \\ 0\end{array}\right],
x_3=\left[\begin{array}{c} 0 \\ 1\end{array}\right] \quad \mathrm{and}\quad
x_4=\left[\begin{array}{c} 1 \\ 1\end{array}\right].
\]
It is easy to see from the table that this quandle is not
abelian; for instance, we have
\[(2\tr 4)\tr (1\tr 2)= 3\tr 1=3 \]
but 
\[(2\tr 1)\tr (4\tr 2)= 2\tr 3=4\ne 3.\]
We also note that this quandle is a kei, so this example also shows that
kei need not be abelian.
}
\end{example}

\begin{lemma}
In addition to right-distributivity, the operation in an abelian quandle
is left-distributive. 
\end{lemma}

\begin{proof}
If $A$ is an abelian quandle, then for all $x,y,z\in A$ we have
\[x\tr(y\tr z) =
(x\tr x)\tr (y\tr z)
=(x\tr y) \tr (x\tr z).\]
\end{proof}

\section{\large \textbf{Hom Quandles}}\label{HQ}

We now wish to study the structure of the set of quandle homomorphisms
$\mathrm{Hom}(Q,A)$ where $Q$ is any quandle and $A$ is an abelian quandle.

%\begin{lemma}
%Define $\tr:\mathrm{Hom}(Q,A)\times \mathrm{Hom}(Q,A)\to \mathrm{Map}(Q,A)$
%by the pointwise operation $(f\tr g)(q)=f(q)\tr g(q)$. Then $f\tr g$ is
%a quandle homomorphism if and only if $A$ is abelian.
%\end{lemma}
%
%\begin{proof}
%Suppose $A$ is an abelian quandle and consider quandle homomorphisms
%$f,g:Q\to A$. Then we have
%\begin{eqnarray*}
%(f\tr g)(x\tr y) & = & f(x\tr y)\tr g(\tr y) \\
%& = & [f(x)\tr f(y)] \tr [g(x)\tr g(y)] \\
%& = & [f(x)\tr g(x)]\tr[f(y)\tr g(y)] \\
%& = & (f\tr g)(x) \tr (f\tr g)(y)
%\end{eqnarray*}
%as required.
%
%\end{proof}

\begin{theorem}
Let $Q$ and $A$ be quandles. If $A$ is abelian, then the set of
quandle homomorphisms $\mathrm{Hom}(Q,A)$ is a quandle under the
pointwise operation $(f\tr g)(q)=f(q)\tr g(q)$.  Moreover, 
$\mathrm{Hom}(Q,A)$ is abelian.
\end{theorem}

\begin{proof}
%We simply verify the axioms.

For Axiom (i), we have
\[(f\tr f)(q)= f(q)\tr f(q)=f(q).\]

For Axiom (ii), define $(f\tr^{-1} g)(q)=f(q)\tr^{-1} g(q)$ in $A$. 
Then we have 
\[((f\tr g)\tr^{-1} g) (q) = (f(q)\tr g(q))\tr^{-1} g(q) = f(q).\]

For Axiom (iii), we have
\begin{eqnarray*}
((f\tr g)\tr h)(q) & = & (f\tr g)(q) \tr h(q) \\
& = & (f(q)\tr g(q))\tr h(q) \\
& = & (f(q)\tr h(q))\tr(g(q)\tr h(q)) \\
& = & (f\tr h)(q)\tr(g\tr h)(q) \\
& = & [(f\tr h)\tr(g\tr h)](q).
\end{eqnarray*}

To show that $\mathrm{Hom}(Q,A)$ is abelian, let 
$f, g, h, k \in \mathrm{Hom}(Q,A)$.  Then

\begin{eqnarray*}
[(f \rhd g) \rhd (h \rhd k)] (q) & = & (f \rhd g) (q) \rhd (h \rhd k) (q) \\
& = & (f(q) \rhd g (q)) \rhd (h(q) \rhd k(q)) \\
& = & (f(q) \rhd h(q)) \rhd (g(q) \rhd k(q)) \\ %\qquad \qquad \textrm{since $A$ is abelian} \\
& = &  (f \rhd h)(q) \rhd (g \rhd k)(q) \\
& = & [(f \rhd h) \rhd (g \rhd k)](q)
 \end{eqnarray*}

\end{proof}

\begin{remark} \textup{Henceforth, we will use $\mathrm{Hom}(Q,A)$ to denote the set of quandle homomorphisms and $\mathrm{\bf Hom}(Q,A)$ to denote the quandle.}

\end{remark}

When the domain quandle $Q$ is a knot quandle, we can interpret 
the pointwise quandle operation in terms of knot diagrams. In particular,
$\mathrm{\bf Hom}(Q(K),A)$ can be represented as the set of 
$A$-labelings of a fixed diagram $D$ of $A$; the $\tr$ operation on diagrams
is then given by the pointwise operation on the arc labels, i.e.
\[\includegraphics{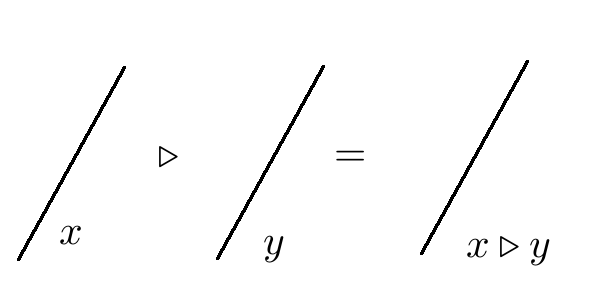}\]

\begin{example}\textup{
In $\mathrm{\bf Hom}(Q(3_1),\mathbb{Z}_3)$ we have
\[\includegraphics{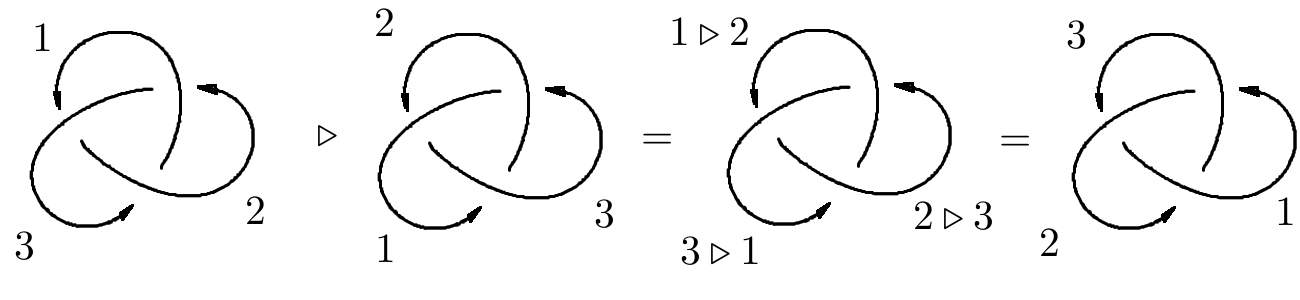}\]
}\end{example}

Then at any crossing, we have
\[\includegraphics{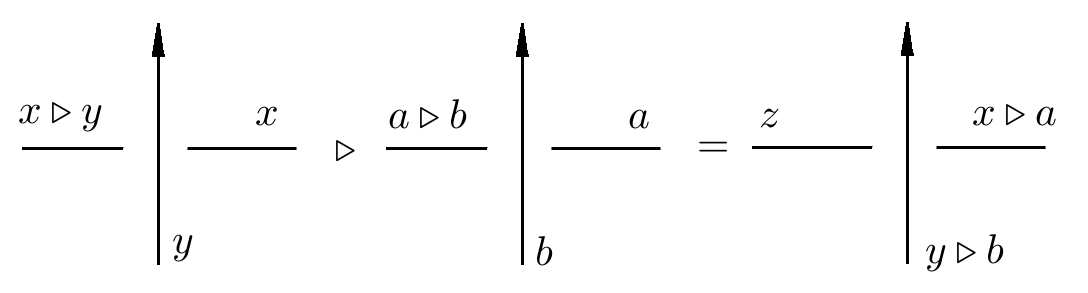}\]
On the one hand, we have $z=(x\tr y)\tr(a\tr b)$ by definition of the $\tr$
operation on diagrams; on the other hand, for the labeling to be a valid
quandle labeling, we must have $z=(x\tr a)\tr(y\tr b)$. Thus, we have

\begin{proposition}\label{p:main}
The set
$\mathrm{Hom}(Q(K),A)$ forms a quandle under the pointwise $\tr$ operation
if and only if $(x\tr y)\tr(a\tr b)=(x\tr a)\tr(y\tr b)$, i.e., if and only if
$A$ is abelian.
\end{proposition}

The hom quandle $\mathrm{\bf Hom}(Q,A)$ inherits many properties held by
the target quandle $A$.

\begin{theorem}
Let $Q$ and $A$ be quandles, where $A$ is abelian. If $A$ is
 commutative or involutory, then $\mathrm{\bf Hom}(Q,A)$ is also commutative or involutory.
 \end{theorem}

\begin{proof}
This follows from a straightforward computation.
\end{proof}

%\begin{theorem}
%Let $Q$ and $A$ be quandles, where $A$ is abelian. If $A$ is
% commutative, then the set of
%quandle homomorphisms $\mathrm{Hom}(Q,A)$ is also commutative.
% \end{theorem}
%
%\begin{proof}
%Let $f, g \in \mathrm{Hom}(Q,A)$ and $q \in Q$.  Then
%\[(f \rhd g)(q) = f(q) \rhd g(q) %\\
% =  g(q) \rhd f(q) %\qquad \qquad \textrm{since $A$ is commutative} \\
% =  (g \rhd f)(q). \]
%%\begin{eqnarray*}
%%(f \rhd g)(q) &=& f(q) \rhd g(q) \\
%%& = & g(q) \rhd f(q) \qquad \qquad \textrm{since $A$ is commutative} \\
%%& = & (g \rhd f)(q)
%%\end{eqnarray*}
%\end{proof}
%
%\begin{theorem} Let $Q$ and $A$ be quandles, where $A$ is abelian. If $A$ is 
%involutory, then the set of
%quandle homomorphisms $\mathrm{Hom}(Q,A)$ is also involutory.
%\end{theorem}
%
%\begin{proof}
%Let $f, g \in \mathrm{Hom}(Q,A)$ and $q \in Q$.  Then
%\[((f\tr g)\tr g)(q)=
%((f(q)\tr g(q))\tr g(q) =
%f(q).
%\]
%\end{proof}

An additional observation:

\begin{theorem}

Let $Q$ be a quandle and $A \cong A'$ be abelian quandles.  Then {\bf Hom}$(Q, A) \cong$ {\bf Hom}$(Q, A')$.

\end{theorem}
\begin{proof}
This follows from a straightforward computation.
\end{proof}

%Recall that a quandle $X$ is \textit{faithful} if the map 
%$f: X\to\mathrm{Aut}(X)$ defined by $f(x)=f_x:X\to X,$ with $f_x(y)=y\tr x,$
%is injective. Then we have

\begin{theorem}
Let $Q$ be a finitely generated quandle and $A$ a finite %faithful 
abelian 
quandle. Then $\mathrm{\bf Hom}(Q,A)$ contains a subquandle isomorphic to $A$.  
\end{theorem}

\begin{proof} Define maps $f_a:Q\to A$ by $f_a(x)=a$ for all $x\in Q$ and
consider the map $\phi : A \rightarrow \mathrm{\bf Hom}(Q,A)$ 
defined by $\phi(a)=f_a$. First,
note that $\phi$ is a homomorphism of quandles since for any $a,b\in A$ we have
\[\phi(a\tr b) = f_{a\tr b}
\quad\mathrm{and}\quad
\phi(a)\tr \phi(b)=f_a\tr f_b\]
Then for any $x\in Q$, we have
\[(f_a\tr f_b)(x)=a\tr b=f_{a\tr b}(x)\]
as required. Further, $\phi$ is injective since $\phi(a)=\phi(b)$ implies
$f_a=f_b$ which implies $a=b$. Then the image 
subquandle 
$\mathrm{Im}(\phi)\subset \mathrm{\bf Hom}(Q,A)$ is isomorphic to $A$.
\end{proof}

%\begin{remark}\textup{The condition that $X$ is faithful is necessary in the
%previous theorem since if $X$ is not faithful, the map $\phi$ may not be
%injective. For instance, consider the Alexander quandles 
%$X=\mathbb{Z}_3[t]/(t-2)$
%and $Y=\mathbb{Z}_4[t]/(t-3)$; we have the following operation matrices:
%\[M_x=\left[\begin{array}{ccc}
%1 & 3 & 2 \\
%3 & 2 & 1 \\
%2 & 1 & 3 \\
%\end{array}\right],\quad
%M_Y=\left[\begin{array}{cccc}
%1 & 3 & 1 & 3 \\
%2 & 4 & 2 & 4 \\
%1 & 3 & 1 & 3 \\
%2 & 4 & 2 & 4 \\
%\end{array}\right],\quad
%\mathrm{and}\quad
%M_{\mathrm{Hom}(X,Y)}=\left[\begin{array}{cc}
%1 & 1 \\
%2 & 2 \\
%\end{array}\right].
%\]
%}\end{remark}

More generally, we have
\begin{theorem}
Let $Q$ be a finitely generated quandle and $A$ an abelian quandle. Then
$\mathrm{\bf Hom}(Q,A)$ is isomorphic to a subquandle of $A^{c}$ where $c$ is
minimal number of generators of $Q$.
\end{theorem}

\begin{proof}
Let $q_1,\dots,q_c$ be a set of generators of $Q$ with minimal cardinality.
Any homomoprhism $f:Q\to A$ must send each $q_k$ to an element $f(q_k)$ in
$A$, and such an assignment of images to generators defines a quandle 
homomorphism if and only if the relations in $Q$ are satisfied by the assignment, i.e.
if and only if
\[f(q_j\tr q_k)=f(q_j)\tr f(q_k)\]
for all $1\le j,k\le c$.
Then the elements of $\mathrm{\bf Hom}(Q,A)$ can be identified with the subset of 
$A^c$ consisting of $c$-tuples of images of generators under $f$ satisfying the
relations in $Q$, i.e.
\[f\leftrightarrow (f(q_1),f(q_2),\dots,f(q_c)).\]
The pointwise operation in $\mathrm{\bf Hom}(Q,A)$ agrees with the componentwise
operation in the Cartesian product $A^c$,  
\begin{eqnarray*}
((f\tr g)(q_1),\dots,(f\tr g)(q_c)) 
& = &  (f(q_1)\tr g(q_1),\dots, f(q_c)\tr g(q_c)) \\
& = &  (f(q_1),\dots, f(q_c)) \tr  (g(q_1),\dots,g(q_c)) 
\end{eqnarray*}
so $\mathrm{\bf Hom}(Q,A)$ is isomorphic to the subquandle of $A^c$ consisting of
$c$-tuples satisfying the relations of $Q$.
\end{proof}

In the simplest case, we can identify the structure of the hom quandle.
Recall that the \textit{trivial quandle of $n$ elements} is a set $T_n$ of 
cardinality $n$ with quandle operation $x\tr y=x$ for all $x,y\in X$.  That is,
the trivial quandle has quandle matrix 
\[M_{T_n}=\left[\begin{array}{rrrr}
1 & 1 & \dots & 1 \\
2 & 2 & \dots & 2 \\
\vdots & \vdots & \ddots & \vdots \\
n & n & \dots & n
\end{array}\right].\]

\begin{lemma}\label{t-hom}
Any map between trivial quandles is a quandle homomorphism.
\end{lemma}

\begin{proof}
Let $f:T_n\to T_m$ be a map between trivial quandles $T_n$ and $T_m$.
Then for any $x_i,x_j\in X$, we have $f(x_i\tr x_j)=f(x_i)$ and
$f(x_i)\tr f(x_j)=f(x_i)$.  Thus 
\[f(x_i\tr x_j)=f(x_i)=f(x_i)\tr f(x_j)\]
and $f$ is a quandle homomoprhism.
\end{proof}

\begin{theorem}  Let $T_n$ and $T_m$ be the trivial quandles of orders $n$ 
and $m$, respectively.  Then $\mathrm{\bf Hom}(T_n, T_m) \cong T_{m^n}$.
\end{theorem}

\begin{proof}
Let $T_n=\{x_1,\dots, x_n\}$ and $T_m=\{y_1,\dots, y_m\}$.
Any map $f:T_n\to T_m$ can be encoded as a vector
\[f=(f(x_1),f(x_2),\dots f(x_n))\]
and there are $m^n$ such maps. Indeed, every such map is quandle
homomorphism by Lemma \ref{t-hom}.

Now, define $\phi:\mathrm{\bf Hom}(T_n,T_m)\to T_{m^n}$ by
\[\phi(y_1,\dots,y_n)=\sum_{k=1}^n y_k m^{k-1}.\]
Then $\phi$ is a bijection; the inverse map rewrites
$x\in \{1,2,\dots,m^n\}$ in base-$m$. 
The quandle structure on $\mathrm{\bf Hom}(T_n,T_m)$ is defined by 
\[(f\tr g)(x)=f(x)\tr g(x)=f(x),\]
so we have $f\tr g=f$ for all $f,g\in\mathrm{\bf Hom}(T_n,T_m)$ and
$\mathrm{\bf Hom}(T_n,T_m)$ is a trivial quandle. Then by Lemma \ref{t-hom},
$\phi$ is an isomorphism of quandles.
\end{proof}

\begin{remark}\textup{
The condition that $\mathrm{\bf Hom}(Q,A)\cong A^c$ where $c$ is the minimal number
of generators of $A$ is not limited to trivial quandles.  For instance, the 
quandle $\mathrm{\bf Hom}(Q(3_1),R_3)$ is isomorphic to $(R_3)^2$, where 
$Q(3_1)$ is the fundamental quandle of the trefoil knot and 
$R_3=\mathbb{Z}_3[t]/(t-2)$ is the connected quandle of 3 elements; we note 
that $Q(3_1)$ has a presentation with two generators (as do all 2-bridge knots).
In the next 
section, we show that $\mathrm{\bf Hom}(Q,A)$ need not be isomorphic to $A^c$
for links with quandle generator index $c$. }
\end{remark}

\section{\large \textbf{Hom Quandle Enhancement}} \label{HQE}

Recall that for any oriented knot $K$ and finite quandle $A$, the cardinality 
of the hom set $\mathrm{Hom}(Q(K),A)$ is a computable knot invariant. As we 
have seen, if $A$ is abelian then the hom set is not just a set but a quandle. 
The natural question is then whether the hom quandle is a stronger invariant
than the counting invariant. In general, an invariant which determines the 
counting invariant is an \textit{enhancement} of the counting invariant, and if
there are examples in which the enhancement distinguishes knots or links which
have the same counting invariant, we say the enhancement is a \textit{proper
enhancement}. Thus, we would like to know whether the hom quandle is a proper 
enhancement. It turns out, the answer is yes:

\begin{example}\textup{Let $A$ be the quandle defined by the quandle matrix
\[M_a=\left[\begin{array}{cccc}
1 & 4 & 4 & 1 \\
3 & 2 & 2 & 3 \\
2 & 3 & 3 & 2 \\
4 & 1 & 1 & 4
\end{array}\right]\]
and consider the links $L6a1$ and $L6a5$ on the Thistlethewaite link table on
the knot atlas \cite{KA}.
\[\begin{array}{cc}
\includegraphics{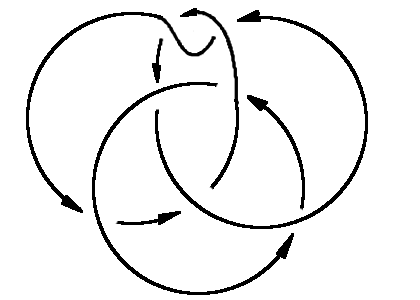} & \includegraphics{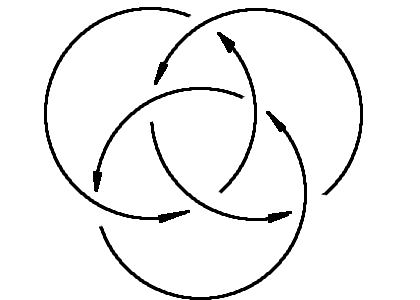} \\
L6a1 & L6a5 \\
\end{array}\]
Our \texttt{Python} computations show that while both hom quandles have the
same cardinality, namely 16, the hom quandles are not isomorphic.
\[M_{\mathrm{\bf Hom}(Q(L6a1),A)}=
\left[\begin{array}{cccccccccccccccc}
1 & 1 & 2 & 2 & 16 & 16 & 15 & 15 & 15 & 15 & 16 & 16 & 2 & 2 & 1 & 1 \\
2 & 2 & 1 & 1 & 15 & 15 & 16 & 16 & 16 & 16 & 15 & 15 & 1 & 1 & 2 & 2 \\
4 & 4 & 3 & 3 & 13 & 13 & 14 & 14 & 14 & 14 & 13 & 13 & 3 & 3 & 4 & 4 \\
3 & 3 & 4 & 4 & 14 & 14 & 13 & 13 & 13 & 13 & 14 & 14 & 4 & 4 & 3 & 3 \\
12 & 12 & 11 & 11 & 5 & 5 & 6 & 6 & 6 & 6 & 5 & 5 & 11 & 11 & 12 & 12 \\
11 & 11 & 12 & 12 & 6 & 6 & 5 & 5 & 5 & 5 & 6 & 6 & 12 & 12 & 11 & 11 \\
9 & 9 & 10 & 10 & 8 & 8 & 7 & 7 & 7 & 7 & 8 & 8 & 10 & 10 & 9 & 9 \\
10 & 10 & 9 & 9 & 7 & 7 & 8 & 8 & 8 & 8 & 7 & 7 & 9 & 9 & 10 & 10 \\
7 & 7 & 8 & 8 & 10 & 10 & 9 & 9 & 9 & 9 & 10 & 10 & 8 & 8 & 7 & 7 \\
8 & 8 & 7 & 7 & 9 & 9 & 10 & 10 & 10 & 10 & 9 & 9 & 7 & 7 & 8 & 8 \\
6 & 6 & 5 & 5 & 11 & 11 & 12 & 12 & 12 & 12 & 11 & 11 & 5 & 5 & 6 & 6 \\
5 & 5 & 6 & 6 & 12 & 12 & 11 & 11 & 11 & 11 & 12 & 12 & 6 & 6 & 5 & 5 \\
14 & 14 & 13 & 13 & 3 & 3 & 4 & 4 & 4 & 4 & 3 & 3 & 13 & 13 & 14 & 14 \\
13 & 13 & 14 & 14 & 4 & 4 & 3 & 3 & 3 & 3 & 4 & 4 & 14 & 14 & 13 & 13 \\
15 & 15 & 16 & 16 & 2 & 2 & 1 & 1 & 1 & 1 & 2 & 2 & 16 & 16 & 15 & 15 \\
16 & 16 & 15 & 15 & 1 & 1 & 2 & 2 & 2 & 2 & 1 & 1 & 15 & 15 & 16 & 16
\end{array}\right]\]
\[M_{\mathrm{\bf Hom}(Q(L6a5),A)}=
\left[\begin{array}{cccccccccccccccc}
1 & 1 & 1 & 1 & 16 & 16 & 16 & 16 & 16 & 16 & 16 & 16 & 1 & 1 & 1 & 1 \\
2 & 2 & 2 & 2 & 15 & 15 & 15 & 15 & 15 & 15 & 15 & 15 & 2 & 2 & 2 & 2 \\
3 & 3 & 3 & 3 & 14 & 14 & 14 & 14 & 14 & 14 & 14 & 14 & 3 & 3 & 3 & 3 \\
4 & 4 & 4 & 4 & 13 & 13 & 13 & 13 & 13 & 13 & 13 & 13 & 4 & 4 & 4 & 4 \\
12 & 12 & 12 & 12 & 5 & 5 & 5 & 5 & 5 & 5 & 5 & 5 & 12 & 12 & 12 & 12 \\
11 & 11 & 11 & 11 & 6 & 6 & 6 & 6 & 6 & 6 & 6 & 6 & 11 & 11 & 11 & 11 \\
10 & 10 & 10 & 10 & 7 & 7 & 7 & 7 & 7 & 7 & 7 & 7 & 10 & 10 & 10 & 10 \\
9 & 9 & 9 & 9 & 8 & 8 & 8 & 8 & 8 & 8 & 8 & 8 & 9 & 9 & 9 & 9 \\
8 & 8 & 8 & 8 & 9 & 9 & 9 & 9 & 9 & 9 & 9 & 9 & 8 & 8 & 8 & 8 \\
7 & 7 & 7 & 7 & 10 & 10 & 10 & 10 & 10 & 10 & 10 & 10 & 7 & 7 & 7 & 7 \\
6 & 6 & 6 & 6 & 11 & 11 & 11 & 11 & 11 & 11 & 11 & 11 & 6 & 6 & 6 & 6 \\
5 & 5 & 5 & 5 & 12 & 12 & 12 & 12 & 12 & 12 & 12 & 12 & 5 & 5 & 5 & 5 \\
13 & 13 & 13 & 13 & 4 & 4 & 4 & 4 & 4 & 4 & 4 & 4 & 13 & 13 & 13 & 13 \\
14 & 14 & 14 & 14 & 3 & 3 & 3 & 3 & 3 & 3 & 3 & 3 & 14 & 14 & 14 & 14 \\
15 & 15 & 15 & 15 & 2 & 2 & 2 & 2 & 2 & 2 & 2 & 2 & 15 & 15 & 15 & 15 \\
16 & 16 & 16 & 16 & 1 & 1 & 1 & 1 & 1 & 1 & 1 & 1 & 16 & 16 & 16 & 16
\end{array}\right]\]
Recall from \cite{QP} that if $X$ is a quandle, then the polyomial
\[\phi(X)=\sum_{x\in X} s^{r(x)}t^{c(x)}\]
where $r(x)=|\{y\in X\ :\ x\tr y=x\}|$ and $c(x)=|\{y\in X\ :\ y\tr x=y\}|$ is
an invariant of quandle isomorphism type known as the \textit{quandle 
polynomial} of $X$. Then we have
\[\phi(\mathrm{\bf Hom}(Q(L6a1),A))= 16s^4t^4\ne
 16s^8t^8=\phi(\mathrm{\bf Hom}(Q(L6a5),A))\]
and the quandles are not isomorphic.
}\end{example}

\section{\large \textbf{Abelian Biquandles}}\label{BQ}

If we think of quandles as the algebraic structure encoding the oriented
Reidemeister moves where the inbound overcrossing arc act on the inbound
undercrossing arc at a crossing, it natural to ask what algebraic structure
results when we allow both inbound \textit{semiarcs}, i.e. portions of the knot divided at both over and undercrossings, to act on each other at a crossing.
The resulting algebraic structure is known as a \textit{biquandle}; see
\cite{K}. More precisely, we have

\begin{definition}\textup{Let $X$ be a set and define $\Delta:X\to X\times X$
by $\Delta(x)=(x,x)$. A \textit{biquandle map}
on $X$ is an invertible map $B:X\times X\to X\times X$ denoted
\[B(x,y)=(B_1(x,y),B_2(x,y))=(y^x,x_y)\]
such that
\begin{itemize}
\item[(i)] There exists a unique invertible \textit{sideways map} $S:X\times X\to X\times X$ such that for all $x,y\in X$, we have
\[S(B_1(x,y),x)=(B_2(x,y),y);\]
\item[(ii)] The component map $(S\Delta)_{1}=(S\Delta)_2:X\to X$ is
bijective, and
\item[(iii)] $B$ satisfies the \textit{set-theoretic Yang-Baxter equation}
\[(B\times I)(I\times B)(B\times I)=(I\times B)(B\times I)(I\times B).\]
\end{itemize} 
A \textit{biquandle} is a set $X$ with a choice of biquandle map $B$.
}\end{definition}

\begin{example}\textup{
Examples of biquandles include
\begin{itemize}
\item \textit{Constant Action Biquandles}. For any set $X$ and bijection
$\sigma:X\to X$, the map $B(x,y)=(\sigma(y),\sigma^{1}(x))$ is a biquandle
map.
\item \textit{Alexander Biquandles.} For any module $X$ over 
$\mathbb{Z}[t^{\pm 1}, r^{\pm 1}]$, the map
\[B(\vec{x},\vec{y})=((1-tr)\vec{x}+t\vec{y},r\vec{x})\]
is a biquandle map.
\item \textit{Fundamental Biquandle of an oriented Link}. Given an 
oriented link diagram $L$, let $G$ be a set of generators corresponding 
bijectively with semiarcs (portions of the link divided at both over and
undercrossing points) in $L$, and define the set of \textit{biquandle words}
$W$ recursively by the rules
\begin{itemize}
\item $G\subset W$ and
\item If $x,y\in W$ then $B_{1,2}^{\pm 1}(x,y),S_{1,2}^{\pm 1}(x,y)\in W$.
\end{itemize}
Then the \textit{fundamental biquandle} of $L$, denoted $B(L)$, is the set of 
equivalence 
classes of $W$ under the equivalence relation generated by the biquandle 
axioms and the \textit{crossing relations} in $L$:
\[\includegraphics{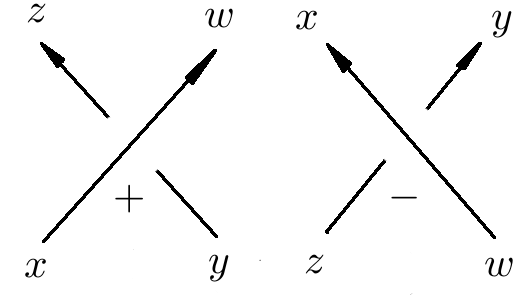}\quad\raisebox{0.5in}{$z=y^x,\ w=x_y$}\]
\item An $n\times 2n$ matrix $M$ with entries in $X=\{1,2,\dots, n\}$ can be 
interpreted as a pair of operation tables, say with $B(i,j)=(M[j,i], M[i,j+n])$.
Then such a matrix defines a biquandle structure on $X$ provided the entries
satisfy the biquandle axioms.
\end{itemize}
}\end{example}

We can generalize the abelian property from quandles to biquandles $X$ by
requiring that the set of biquandle homomorphsims $\mathrm{Hom}:B(L)\to X$
forms a biquandle under the diagrammatic operations analogous to the quandle 
case. 
\[\includegraphics{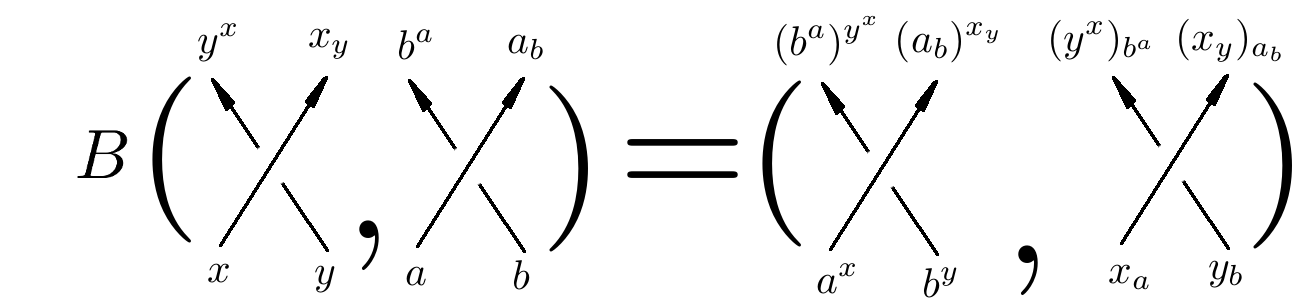}\]
More precisely, we have
\begin{definition}\textup{
We say a biquandle map $B:X\times X\to X\times X$ is \textit{abelian} if for 
all $a,b,x,y\in X$ we have
\[(b^a)^{y^x}=(b^y)^{a^x}, \quad
(a_b)^{x_y}=(a^x)_{b^y}, \quad
%(y^x)_{b^a}=(y_b)^{x_a},\quad
\quad\mathrm{and}\quad
(x_y)_{a_b}=(x_a)_{y_b}.
\]
Note that two of the four conditions determined in the diagram, namely
$(a_b)^{x_y}=(a^x)_{b^y}$ and $(y^x)_{b^a}=(y_b)^{x_a}$, are equivalent.
}\end{definition}

\begin{example}\textup{
Alexander biquandles are abelian, as we can verify directly.
\begin{eqnarray*}
(b^a)^{y^x} & = & t(b^a)+(1-tr)(y^x) \\
& = &  t(tb+(1-tr)a)+ (1-tr)(ty+(1-tr)x) \\
& = &  t^2b+t(1-tr)a+ t(1-tr)y+(1-tr)^2x \\
& = &  t^2b+t(1-tr)y+ t(1-tr)a+(1-tr)^2x \\
& = & (b^y)^{a^x}, \\
%\end{eqnarray*}
%\begin{eqnarray*}
& & \\
(a_b)^{x_y} & = & t(a_b)+(1-tr)(x_y) \\
& = & tra+(1-tr)(rx) \\
& = & r(ta+(1-tr)x) \\
& = & (a^x)_{b^y},
\end{eqnarray*}
and
\[(x_y)_{a_b}=r(x_y)=r^2x=r(x_a)=(x_a)_{y_b}.\]
}\end{example}

As with quandles, we have
\begin{proposition}
If $Y$ is a biquandle and  $X$ is an abelian biquandle, then the set of 
biquandle homomorphisms
$\mathrm{Hom}(Y,X)$ has a biquandle structure defined by
\[f^g(x)= f(x)^{g(x)}\quad\mathrm{and}\quad  g_f(x)=g(x)_{f(x)}.\]
\end{proposition}

In particular, if $X$ is a finite abelian biquandle, then the hom
biquandle $\mathrm{\bf Hom}(B(L),X)$ is a link invariant which determines,
but is not determined by, the biquandle counting invariant.

\begin{example}\textup{
Consider the biquandle with operation matrix
\[M_X=\left[\begin{array}{ccccc|ccccc}
3 & 3 & 3 & 1 & 1 & 1 & 1 & 1 & 1 & 1 \\
1 & 1 & 1 & 2 & 2 & 2 & 2 & 2 & 2 & 2 \\
2 & 2 & 2 & 3 & 3 & 3 & 3 & 3 & 3 & 3 \\
5 & 5 & 5 & 5 & 5 & 4 & 4 & 4 & 5 & 5 \\
4 & 4 & 4 & 4 & 4 & 5 & 5 & 5 & 4 & 4 \\
\end{array}\right].\]
Our \texttt{Python} computations indicate that both the $(4,2)$-torus link
$L4a1$ and the Whitehead link $L5a1$ have biquandle counting invariant
value $|\mathrm{Hom}(B(L4a1),X)|=81=|\mathrm{Hom}(B(L5a1),X)|$ with respect
to $X$, but $\mathrm{\bf Hom}(B(L4a1),X)$ and $\mathrm{\bf Hom}(B(L5a1),X)$ are not 
isomorphic as biquandles.  Indeed, they have distinct upper biquandle 
polynomials
\[\phi_1(\mathrm{\bf Hom}(B(L4a1),X))=12s^{77}t^{77} + 4s^{77}t^{68} + 40s^{71}t^{71} + 9s^{68}t^{72} + 16s^{65}t^{65}\]
and \[\phi_1(\mathrm{\bf Hom}(B(L5a1),X))=16s^{77}t^{77} + 40s^{71}t^{71} + 4s^{65}t^{56} + 9s^{56}t^{60} + 12s^{56}t^{56}\]
respectively where
\[\phi_1(X)=\sum_{x\in X} s^{r(x)}t^{c(x)}\]
where $r(x)=|\{y\in X\ :\ x^y=x\}|$ and $c(x)=|\{y\in X\ :\ y_x=y\}|$.
\[
\begin{array}{cc}
\includegraphics{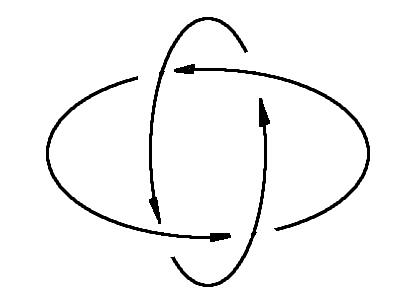} & \includegraphics{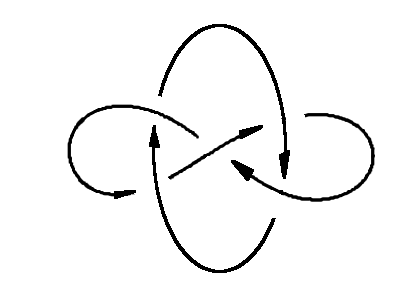} \\
L4a1 & L5a1 \\
\end{array}
\]
}\end{example}

\section{\large \textbf{Categorical Framework}}\label{C}

Since we know that the collection of homomorphisms between two abelian quandles forms an abelian quandle, a natural question to ask is whether the category of abelian quandles is a `symmetric monoidal closed' category.  In this section, we show that the answer to this question is yes.

Recall that a symmetric monoidal category $C$ is \textit{closed} if for any object $Q \in C$, the functor $ - \otimes Q: C \rightarrow C$ has a right adjoint.  We will denote this right adjoint by $\mathrm{\bf Hom}(Q,-)$ and the adjointness condition then says:
$$\textrm{Hom}(P \otimes Q, R) = \textrm{Hom}(P, \mathrm{\bf Hom}(Q,R)) \qquad (\ast) $$
for all $P,Q,R \in C$ and where ``=" means natural isomorphism.  The right adjoint $\mathrm{\bf Hom}(-,-)$ is uniquely determined by adjointness and defines a functor $C^{op} \times C \rightarrow C$.

\subsection{Tensor Product of Abelian Quandles}

Let $Q$ and $A$ be abelian quandles.  We define the {\it tensor product}, $Q \otimes A$, to be the free quandle on the set $Q \times A$ quotiented out by the relations 
$$(q, a_1) \rhd (q, a_2) = (q, a_1 \rhd a_2) \qquad \textrm{and} \qquad (q_1, a) \rhd (q_2, a) = (q_1 \rhd q_2, a)$$
for $q, q_1, q_2 \in Q$ and $a, a_1, a_2 \in A$.  For abelian quandles $X$, a homomorphism $Q \otimes A \rightarrow X$ is essentially the same thing as a \textit{bihomomorphism}, that is, a function $f: Q \times A \rightarrow X$ such that:
\begin{itemize}
\item $f(q, -): A \rightarrow X$ is a homomorphism for each $q \in Q$, and
\item $f(-, a): Q \rightarrow X$ is a homomorphism for each $a \in A$.

\end{itemize}
The unit for this tensor product is the one-element quandle $1$, which can be checked directly.  We note that, in principle, the unit is actually the free quandle on a single generator, but that is, in fact, the one-element quandle due to the first quandle axiom.

We remark that this situation works very much as it does for modules.  The key point is that both the theory of abelian quandles and the theory of modules are `commutative' theories.  We recall that a \textit{commutative theory} is an algebraic theory such that each operation of the theory is a homomorphism.  For example, the theory of abelian groups is commutative because for any abelian group $G$, the map $+: G \times G \rightarrow G$ is a homomorphism.  More explicitly, in a commutative theory, given any $n$-ary operation $\alpha$ and any $m$-ary operation $\beta$, the equation
$$\alpha(\beta(x_{11}, \ldots, x_{1m}), \ldots , \beta(x_{n1}, \ldots, x_{nm})) = \beta(\alpha(x_{11}, \ldots, x_{n1}), \ldots , \alpha(x_{1m}, \ldots, x_{nm}))$$
holds.  Since the theory of quandles only has one operation $\rhd$, all we need is the equation
$$(x_{11} \rhd x_{12}) \rhd (x_{21} \rhd x_{22}) = (x_{11} \rhd x_{21}) \rhd (x_{12} \rhd x_{22}),$$
and this is precisely what the definition of abelian quandle guarantees.

\subsection{Category of Abelian Quandles}
By Theorem \ref{p:main}, we know that given abelian quandles $Q$ and $A$,  the set Hom$(Q,A)$ becomes an abelian quandle under pointwise operations.  Indeed, $\textrm{Hom}(Q, A)$ is the underlying set of $\textrm{\bf Hom}(Q,A)$.  It is not difficult to show that a homomorphism $Q \otimes A \rightarrow X$ is essentially the same thing as a homomorphism $Q \rightarrow \textrm{\bf Hom}(A, X)$.  This fact, more formally, gives us the answer to the question raised at the beginning of this section:

\begin{theorem}
The category of abelian quandles is symmetric monoidal closed under the tensor product $\otimes$ and closed structure \textup{\textbf{Hom}(-,-)} defined above.
\end{theorem}

\begin{proof}
This follows from the main theorem of Linton \cite{L} since the theory of abelian quandles is commutative.
\end{proof}

We note that this means the category of abelian quandles can be enriched over itself, and the adjointness condition $(\ast)$ is a natural isomorphism of quandles
$$\textrm{\bf Hom}(Q \otimes A, X) = \textrm{\bf Hom}(Q, \textrm{\bf Hom}(A,X)).$$.

\section{\large \textbf{Questions for Future Research}}\label{Q}

In this section we collect a few questions for future research.

In \cite{J}, Joyce shows that the fundamental abelian quandle of a classical
knot determines, and is determined by, the fundamental Alexander quandle of 
the knot. Is the analogous statement true for Alexander biquandles?

What other properties does the hom quandle, Hom$(Q,A)$ inherit from the quandles $Q$ and $A$?  For example, does it inherit the properties of being connected, dihedral, Core or conjugation?  Given connected quandles, is it true that the hom quandle structure is determined by the counting invariant?  

Moreover, what is the relationship (if any) between the cardinalities of $Q$, $A$, and Hom$(Q,A)$?  How is the homology of $Q \otimes A$ related to the homologies of $Q$ and $A$?  Under the conjugation or Core functors, the braid group becomes a quandle.  Is this `braid quandle' abelian?

\bigskip

\noindent
\textsc{Department of Mathematics\\
Loyola Marymount University\\
One LMU Drive, Suite 2700\\
Los Angeles, CA 90045}

%\bigskip

%\noindent
%\textsc{School of Mathematics \\ 
%University of Edinburgh \\ 
%Edinburgh EH9 3JZ, UK}

\bigskip

\noindent
\textsc{Department of Mathematical Sciences\\ 
Claremont McKenna College \\
850 Colubmia Ave.\\
Claremont, CA 91711}

\end{document}